\documentclass{amsart}
\usepackage{graphicx}
\usepackage{amssymb,amsmath,amsfonts,amsthm,amsxtra,setspace,mathrsfs}
\usepackage{hyperref}
\usepackage{enumitem}
\usepackage{graphics}
\usepackage{comment}
\usepackage{mathtools}
\usepackage{harpoon}

\usepackage{tikz-cd}
\usepackage{textcomp}
\usepackage[pagewise]{lineno}

\newtheorem{thm}{Theorem}[section]

\newtheorem{cor}[thm]{Corollary}
\newtheorem{lem}[thm]{Lemma}
\newtheorem{prop}[thm]{Proposition}
\newtheorem{prob}[thm]{Problem}
\newtheorem{fact}[thm]{Fact}
\newtheorem{question}[thm]{Question}
\theoremstyle{definition}
\newtheorem{defn}[thm]{Definition}
\newtheorem{rem}[thm]{Remark}

\renewcommand{\epsilon}{\varepsilon}
\renewcommand{\phi}{\varphi}

\newcommand{\op}[1]{\operatorname{#1}}

\newcommand{\norm}[1]{\left\Vert #1 \right\Vert}
\newcommand{\abs}[1]{\left\vert #1 \right\vert}

\title{$K$-theory of co-existentially closed continua}
\author[C. J. Eagle]{Christopher J. Eagle${}^1$} 
\address[C. J. Eagle]{University of Victoria, Department of Mathematics and Statistics, PO BOX 1700 STN CSC, Victoria, British Columbia, Canada, V8W 2Y2}%
\email{eaglec@uvic.ca}
\urladdr{http://www.math.uvic.ca/~eaglec}
\thanks{${}^1$ Supported by NSERC Discovery Grant RGPIN-2021-02459}

\author[J. Lau]{Joshua Lau${}^{2}$} 
\address[J. Lau]{University of Toronto, Department of Mathematics, 40 St. George St., Toronto, Ontario, Canada, M5S 2E4}%
\email{jlau@math.utoronto.ca}
\urladdr{}
\thanks{${}^2$ Supported by an NSERC USRA}
\date{\today}

\begin{document}

\begin{abstract}
We describe the possible values of $K$-theory for $C(X)$ when $X$ is a co-existentially closed continuum.  As a consequence we also show that all pseudo-solenoids, except perhaps the universal one, are not co-existentially closed.
\end{abstract}

\maketitle

\section{Introduction}
A \emph{compactum} is a compact Hausdorff space and a \emph{continuum} is a connected compactum (not necessarily metrizable).  Although model-theoretic methods cannot be directly applied to compacta or continua, there has nevertheless been an extensive study of model-theoretic properties of compacta in ``dual" form; see \cite{Bankston1987} and subsequent papers.  When $X$ is $0$-dimensional one can replace the study of $X$ by the study of the Boolean algebra of clopen subsets of $X$, due to Stone duality.  For more general compacta one can use continuous model theory to study commutative unital C*-algebras, which are dual to the category of compacta by Gel'fand duality.  This is the approach taken in this paper.

Bankston \cite{Bankston1999} introduced co-existentially continua as the dual of the model-theoretic notion of existentially closed models of a theory (see Section \ref{sec:Preliminaries} below for the definition).  While co-existentially closed continua exist in abundance, only one concrete example is known.  In \cite{EagleGoldbringVignati} it was shown that the \emph{pseudoarc} is a co-existentially closed continuum.  The pseudoarc was first constructed by Knaster \cite{Knaster} and later characterized by Bing \cite{Bing1951} as the unique non-degenerate hereditarily indecomposable chainable continuum; Bing \cite{Bing1951} also showed that the pseudoarc is generic amongst subcontinua of $\mathbb{R}^n$ for every $n \geq 2$.  In \cite{EagleGoldbringVignati} it was shown that if the theory of continua, $T_{\op{conn}}$, has a model companion then that model companion must be the theory of the pseudoarc, and moreover in that case the theory of the pseudoarc is exactly the common theory of $C(X)$ for co-existentially closed continua $X$.  It remains an open question whether or not $T_{\op{conn}}$ has a model companion.

In this paper we study the $K$-theory of $C(X)$ when $X$ is a co-existentially closed continuum.  Our main results are encapsulated in the following description:

\begin{thm}\label{thm:MainThm}
Let $X$ be a co-existentially closed continuum.  Then $K_0(C(X)) = \mathbb{Z}$ and $K_1(C(X))$ is a torsion-free divisible abelian group that may have arbitrarily large rank.
\end{thm}

The remainder of the paper is organized as follows.  In Section \ref{sec:Preliminaries} we recall some preliminary facts we will need from the model theory of continua.  Section \ref{sec:K0} contains the results about $K_0(C(X))$, which in fact only require that $\dim(X) = 1$.  Section \ref{sec:K1} contains the results relating to $K_1(C(X))$, as well as a proof that planar co-existentially closed continua do not separate the plane.  Finally, Section \ref{sec:PseudoSolenoid} shows that various familiar continua are not co-existentially closed.  In particular, we show that the only homogeneous planar co-existentially closed continuum is the pseudoarc, and we also show that all solenoids and pseudo-solenoids, except possibly the universal pseudo-solenoid, fail to be co-existentially closed.

\subsection*{Acknowledgements}
We are grateful to George Elliott, Heath Emerson, and Isaac Goldbring for many helpful discussions and comments.  We also appreciate helpful comments from the anonymous referee.  Part of this work was completed during the 2023 Thematic Program on Operator Algebras and Applications at the Fields Institute; we thank the Fields Institute for providing an excellent environment for collaboration.

\section{Preliminaries}\label{sec:Preliminaries}
This paper concerns model-theoretic properties of continua.  Model theory does not directly apply to topological spaces, so given a compactum $X$ we treat it model-theoretically by instead studying the C*-algebra $C(X)$ of continuous complex-valued functions on $X$; through Gel'fand duality the categories of compacta and of unital commutative C*-algebras are equivalent.  To treat $C(X)$ model-theoretically we use continuous model theory, as developed in \cite{BenYaacovEtAl}; the reader looking for background of continuous model theory, particularly in the context of C*-algebras, is referred to \cite{FarahEtAl}.  We also assume that the reader is familiar with the elements of continuum theory (for which we suggest \cite{Nadler}) and C*-algebra $K$-theory (for which see \cite{KTheoryBook}).

\begin{fact}[{\cite[Fact 1.1 and Remark 1.2]{EagleGoldbringVignati}}]
There is a universally axiomatizable theory $T_{\op{conn}}$ in the language of unital C*-algebras such that $M \models T_{\op{conn}}$ if and only if $M \cong C(X)$ for some continuum $X$.
\end{fact}

The ultraproduct construction for commutative unital C*-algebras dualizes to the \emph{ultracoproduct} of compact Hausdorff spaces, which we denote $\sum_{\mathcal{U}}X_i$.  For our purposes it will suffice for us to know that $\sum_{\mathcal{U}}X_i$ is the spectrum of the C*-algebra ultraproduct $\prod_{\mathcal{U}}C(X_i)$, so we omit the (somewhat lengthy) hands-on definition of $\sum_{\mathcal{U}}X_i$; see \cite{Bankston1987}.  Banskton \cite{Bankston1999} introduces co-existentially closed continua in terms of mapping properties involving ultracopowers, but for our purposes it is more convenient to define a continuum $X$ to be co-existentially closed if $C(X)$ is an existentially closed model of $T_{\op{conn}}$.  We recall the precise definitions of existential closure in the setting of continuous logic:

\begin{defn}
Let $T$ be a theory in a signature $L$ of continuous first-order logic.  A model $M \models T$ is called an \emph{existentially closed model of $T$} if given any $N \models T$ such that $M \subseteq N$, any quantifier-free $L$-formula $\phi(\overline{x}, \overline{y})$, and any tuple $\overline{b}$ from $M$ (of the appropriate length), we have 
\[\left(\inf_{\overline{x}}\phi(\overline{x}, \overline{b})\right)^M = \left(\inf_{\overline{x}}\phi(\overline{x}, \overline{b})\right)^N.\]

A continuum $X$ is \emph{co-existentially closed} if $C(X)$ is an existentially closed model of $T_{\op{conn}}$.
\end{defn}

The equivalence of Bankston's definition of co-existentially closed continua to the one we give here can be found in the appendix to \cite{EagleGoldbringVignati}.  

The next result was first proved by Bankston \cite{Bankston2000} using lattice bases, but can also be obtained by applying standard model-theoretic results about universally axiomatizable theories to $T_{\op{conn}}$ and then using Gel'fand duality.

\begin{fact}\label{fact:ContinuousImage}
Every continuum is the continuous image of a co-existentially closed continuum of the same weight.
\end{fact}

It will be useful to know that if $X$ is a co-existentially closed continuum then $\op{Th}_{\forall\exists}(C(X))$ is the maximal consistent $\forall\exists$-theory extending $T_{\op{conn}}$ for non-degenerate continua.  The proof relies on the following fact, originally due to K. P. Hart in \cite{Hart}; see \cite{EagleGoldbringVignati} for the details of how to translate the result from \cite{Hart} into the version we state here.

\begin{fact}
Suppose that $X$ and $Y$ are continua and that $Y$ is non-degenerate.  Then there is a continuum $Y'$ such that $C(Y') \equiv C(Y)$ and $C(X)$ embeds into $C(Y')$.
\end{fact}

\begin{lem}\label{lem:UniversalAETheory}
Let $X$ be a co-existentially closed continuum.  If $\sigma$ is a $\forall\exists$-sentence and there is a non-degenerate continuum $Y$ such that $C(Y) \models \sigma$ then $C(X) \models \sigma$.
\end{lem}
\begin{proof}
Apply the fact above to find a continuum $Y'$ with $C(Y') \equiv C(Y)$ and such that $C(X)$ embeds in $C(Y')$; by replacing $C(X)$ by an isomorphic copy we may assume that in fact $C(X)$ is a substructure of $C(Y')$.  Write $\sigma$ as $\sup_{\norm{\overline{x}} \leq R}\phi(\overline{x})$, where $\phi$ is an existential formula.  Fix any tuple $\overline{a}$ from $C(X)$ with $\norm{\overline{a}} \leq R$.  Since $C(Y') \equiv C(Y)$ and $C(Y) \models \sigma$ we have $C(Y') \models \sigma$, and since $C(X)$ is a substructure of $C(Y')$ the tuple $\overline{a}$ is also in $C(Y')$, so $\phi(\overline{a})^{C(Y')} = 0$.  Now $C(X)$ is existentially closed in $C(Y')$, so $\phi(\overline{a})^{C(X)} = 0$ as well.  Since $\overline{a}$ was arbitrary, this shows that $C(X) \models \sigma$.
\end{proof}


Throughout this paper when we refer to the \emph{dimension} of a compactum we mean the covering dimension. 

\begin{fact}[{\cite[Corollary 4.13]{Bankston2006}}]\label{fact:HIDIM}
Let $X$ be a co-existentially closed continuum.  Then $X$ is hereditarily indecomposable and $\dim(X) = 1$.
\end{fact}

We denote by $\mathbb{T}$ the circle, which we often view as $\mathbb{T} = \{z \in \mathbb{C} :\abs{z} = 1\}$.

\section{$K_0$ for co-existentially closed continua}\label{sec:K0}
In this section we show that if $X$ is a co-existentially closed continuum then $K_0(C(X)) = \mathbb{Z}$.  In fact, we show that if $X$ is any $1$-dimensional continuum then $K_0(C(X)) = \mathbb{Z}$; we suspect that this result may already have been known, but as we were unable to locate a reference in the literature, we provide a proof here.

\begin{fact}[{\cite[Example 3.3.5]{KTheoryBook}}]
Let $X$ be any continuum.  There is a surjective group homomorphism $D : K_0(C(X)) \to \mathbb{Z}$ that satisfies $D([p]) = \op{Tr}(p(x))$ (independently of the choice of $x \in X$).  Here we identify $M_n(C(X))$ with $C(X, M_n(\mathbb{C}))$ and $\op{Tr}$ is the standard trace on $M_n(\mathbb{C})$.
\end{fact}

The map $D$ in the previous fact is also often called $\dim$, but to avoid confusion we reserve $\dim$ for the covering dimension of a compactum.

Note that $K_0(C(X)) = \mathbb{Z}$ if and only if $D$ is an isomorphism, in which case $K_0(C(X))$ is generated by the class $[1_X]$, where $1_X : X \to \mathbb{C}$ is the function that is constantly $1$.

For each $n \in \mathbb{N}$, let $Z_n$ denote the wedge product of $n$ circles at a common basepoint, and let $Z_0$ be a single point.  We will need the following two facts:

\begin{fact}[{\cite[Example 1.22 and Section 1.A]{Hatcher}}]\label{fact:HomotopyCircles}
Every $1$-dimensional compact connected CW-complex is homotopy equivalent to $Z_n$ for some $n$.
\end{fact}
\begin{fact}[{\cite[Exercise 12.3]{KTheoryBook}}]\label{fact:K0Circles}
For every $n \in \mathbb{N}$, $K_0(C(Z_n)) = \mathbb{Z}$.
\end{fact}

\begin{prop}\label{prop:CWK0}
    Let $X$ be a compact connected CW-complex with $\dim(X) \leq 1$.  Then $K_0(C(X)) = \mathbb{Z}$.
\end{prop}
\begin{proof}
		If $\dim X = 0$ then $X$ consists of a single point, so $K_0(C(X)) = K_0(\mathbb{C}) = \mathbb{Z}$.
  
		If $\dim X = 1$, then by Fact \ref{fact:HomotopyCircles} $X$ is homotopy equivalent to $Z_n$ for some $n \geq 0$.  Since $K_0$ is invariant under homotopy equivalence, when $n=0$ we are back in the previous case, while if $n \geq 1$ then by Fact \ref{fact:K0Circles} we have $K_0(C(X)) = K_0(C(Z_n)) = \mathbb{Z}$. 
\end{proof}

\begin{prop}\label{prop:InverseLimit}
Let $X_1 \leftarrow X_2 \leftarrow \cdots$
be an inverse sequence of connected compact CW-complexes, each of dimension at most $1$. Then $K_0\left(C\left(\varprojlim X_i\right)\right) = \mathbb{Z}$.
\end{prop}
\begin{proof}
The inverse sequence of spaces induces a sequence of commutative unital C*-algebras and unital $*$-homomorphisms

		\[\begin{tikzcd}
			{C(X_1)} & {C(X_2)} & {C(X_3)} & \cdots
			\arrow["{\varphi_1}", from=1-1, to=1-2]
			\arrow["{\varphi_2}", from=1-2, to=1-3]
			\arrow["{\varphi_3}", from=1-3, to=1-4]
		\end{tikzcd}\]
		By Gel'fand duality we have
		\begin{align*}
			\varinjlim C(X_n) = C\left(\varprojlim X_n\right),
		\end{align*}
		and by continuity of the $K_0$ functor we have that
		\begin{align*}
			K_0\left(C\left(\varprojlim X_n\right)\right) = K_0\left(\varinjlim C(X_n)\right) = \varinjlim K_0(C(X_n)).
		\end{align*}
		Now since $\varphi_n : C(X_n) \to C(X_{n+1})$ is unital for every $n$, it follows that the induced group homomorphism $K_0(\varphi_n) : K_0(C(X_n)) \to K_0(C(X_{n+1}))$ satisfies $K_0(\varphi_n)([1_{X_n}]) = [\varphi_n(1_{X_n})] = [1_{X_{n+1}}]$. By Proposition \ref{prop:CWK0}, $K_0(C(X_i)) = \mathbb{Z}$, generated by $[1_{X_i}]$.  It follows that $K_0(\varphi_n)$ is an isomorphism for every $n$, and hence
		\begin{align*}
			K_0(C(\varprojlim X_n)) = \varinjlim K_0(C(X_n)) = K_0(C(X_1)) = \mathbb{Z}
		\end{align*}
\end{proof}

\begin{thm}\label{thm:Dim1K0}
    If $X$ is a continuum of dimension $1$ then $K_0(C(X)) = \mathbb{Z}$.
\end{thm}
\begin{proof}
		It is shown in \cite[Corollary 1]{mardesic} that we can express $X$ as 
		\begin{align*}
			X = \varprojlim_{b} \left( \varprojlim_{i \in \mathbb{N}} P_{b,i} \right)
		\end{align*}
		for some compact polyhedra $P_{b,i}$ each having $\dim P_{b,i} \leq \dim X = 1$ (if $X$ is non-metrizable this limit may be over an uncountable indexing family).  Since $X$ maps continuously onto each term of the inverse system, and $X$ is connected, each $P_{b, i}$ is also connected.  Thus each $P_{b,i}$ is, in particular, a connected compact CW-complex of dimension at most $1$.  For each fixed $b$, let $X_b = \displaystyle{\varprojlim_{i}} P_{b,i}$.   By Proposition \ref{prop:InverseLimit} we obtain, for each $b$, that
		\begin{align*}
			K_0(C(X_b)) = K_0\left(C\left(\varprojlim_{i} P_{b,i}\right)\right) = \mathbb{Z}.
		\end{align*}
        For every $b$ and $b'$ the $*$-homomorphism $C(X_b) \to C(X_{b'})$ is unital, so it follows that the induced group homomorphism $K_0(C(X_b)) \to K_0(C(X_{b'}))$ is an isomorphism.  Therefore we have:
        \[
			\varinjlim_b K_0(C(X_b)) = \mathbb{Z}
        \]
		Finally, by continuity of $K_0$ we obtain: 
		\begin{align*}
			K_0(C(X)) = K_0\left(C\left(\varprojlim_b X_b\right)\right) = K_0\left(\varinjlim_b C(X_b)\right) = \varinjlim_b K_0(C(X_b)) = \mathbb{Z}
		\end{align*}
\end{proof}

Finally, if $X$ is a co-existentially closed continuum then Fact \ref{fact:HIDIM} says that $X$ has dimension $1$, so we obtain:

\begin{cor}\label{cor:coecK0}
Every co-existentially closed continuum $X$ has $K_0(C(X)) = \mathbb{Z}$.
\end{cor}

\section{$K_1$ for co-existentially closed continua}\label{sec:K1}
Unlike the case for $K_0$, when $X$ is a co-existentially closed continuum there are many possible values of $K_1(C(X))$.  In this section we show that $K_1(C(X))$ must be a torsion-free divisible abelian group, but also show that the rank of $K_1(C(X))$ can be arbitrarily large when $X$ is a co-existentially closed continuum.

In fact, we will work primarily with the \emph{first integral \v{C}ech cohomology group of $X$}, $\check{H}^1(X)$, in order to take advantage of existing results in the literature.  Recall that $\check{H}^1(X)$ can be identified with the collection of homotopy classes of continuous maps from $X$ to the circle $\mathbb{T}$ (see \cite[Exercise 4.3.2]{Hatcher}).  The connection to $K$-theory for $C(X)$ is the following:

\begin{prop}\label{prop:K1H1}
Let $X$ be a continuum with $\dim(X) = 1$.  Then $K_1(C(X)) \cong \check{H}^1(X)$.
\end{prop}
\begin{proof}
By \cite[V.3.1.3]{Blackadar}, since $\dim(X) \leq 1$ the stable rank of $C(X)$ is given by
\[\op{sr}(C(X)) = \left\lfloor \frac{\dim(X)}{2}\right\rfloor + 1 = 1.\]
Therefore by \cite[V.3.1.26]{Blackadar} we have $K_1(C(X)) \cong U(C(X)) / (U_{\circ}(C(X))$.  That is,
\begin{align*}
K_1(C(X)) &\cong U(C(X)) / U_{\circ}(C(X)) \\
&= C(X, \mathbb{T}) / \text{homotopy equivalence} \\
&= \check{H}^1(X)
\end{align*}
\end{proof}

In particular, applying Fact \ref{fact:HIDIM} we have:

\begin{cor}
If $X$ is a co-existentially closed continuum then $K_1(C(X)) \cong \check{H}^1(X)$.
\end{cor}

We make some definitions, motivated by the analogy between existentially closed structures and algebraically closed fields.  Later in this section we will use these notions to show that when $X$ is co-existentially closed, $K_1(C(X))$ is a divisible group.

\begin{defn}
A C*-algebra $A$ is \emph{approximately algebraically closed} if for every $n > 0$, every $a_0, \ldots, a_{n-1} \in A$, and every $\epsilon > 0$, there is $f \in A$ such that $\norm{a_0+a_1f+\cdots+a_{n-1}f^{n-1}+f^n} < \epsilon$.  We say $A$ is \emph{algebraically closed} if, in the same situation, $f$ can always be found such that $a_0+a_1f+\cdots+a_{n-1}f^{n-1}+f^n=0$.
\end{defn}

We cannot axiomatize $C(X)$ being algebraically closed in continuous logic (see Corollary \ref{cor:AlgClosedNotAxiomatizable} below).  We can, however, axiomatize being approximately algebraically closed, which we now do.

For each $n \in \mathbb{N}_{> 0}$ and each $K \in \mathbb{R}_{>0}$, define
\[\sigma_{n,K}: \sup_{\norm{a_0} \leq K} \cdots \sup_{\norm{a_{n-1}} \leq K}\inf_{\norm{f} \leq 2+K}\norm{a_0+a_1f+\cdots+a_{n-1}f^{n-1}+f^n}.\]

\begin{prop}\label{prop:AlgClosedApprox}
For any compactum $X$, the following are equivalent:
\begin{enumerate}
    \item{$C(X)$ is approximately algebraically closed.}
    \item{$C(X) \models \sigma_{n, K}$ for all $n$ and $K$.}
\end{enumerate}
\end{prop}
\begin{proof}
Suppose that (1) holds.  Fix $n$ and $K$.  Suppose that $a_0, \ldots, a_{n-1} \in C(X)$ are such that $\norm{a_j} \leq K$ for all $j$.  Consider any $\epsilon$ with $0 < \epsilon < 1$.  Using (1), find $f \in C(X)$ such that $\norm{a_0+a_1f+\cdots+a_{n-1}f^{n-1}+f^n} < \epsilon$.  It suffices to show that $\norm{f} \leq 2+K$.  Fix any $x \in X$, and write
\[a_0(x)+a_1(x)f(x) + \cdots + a_{n-1}(x)f(x)^{n-1} + f(x)^n = Re^{i\theta}.\]
Note that $\norm{a_0+a_1f+\cdots+a_{n-1}f^{n-1}+f^n} < \epsilon$ precisely means $0 \leq R < \epsilon$.  Then
\[(a_0(x) - Re^{i\theta}) + a_1(x)f(x) + \cdots + a_{n-1}f(x)^{n-1}+f(x)^n = 0.\]
By Cauchy's bound for roots of complex polynomials, we have
\begin{align*}
\abs{f(x)} &\leq 1+\max\{\abs{a_0(x)-Re^{i\theta}}, \abs{a_1(x)}, \ldots, \abs{a_{n-1}(x)}\} \\
&\leq 1 + \max\{\abs{a_0(x)}+R, \abs{a_1(x)}, \ldots, \abs{a_{n-1}(x)}\} \\
&\leq 1+R+\max\{\abs{a_0(x)}, \ldots, \abs{a_{n-1}(x)}\} \\
&< 1+\epsilon+\max\{\abs{a_0(x)}, \ldots, \abs{a_{n-1}(x)}\} \\
&< 2+\max\{\abs{a_0(x)}, \ldots, \abs{a_{n-1}(x)}\} \\
&\leq 2+\max\{\norm{a_0}, \ldots, \norm{a_{n-1}}\} \\
&\leq 2+K
\end{align*}
As this holds for every $x \in X$, we have $\norm{f} \leq 2+K$, as desired to show that $C(X) \models \sigma_{n, K}$.

Now suppose that (2) holds.  Given any $a_0, \ldots, a_{n-1} \in C(X)$, let $K = \max\{\norm{a_0}, \ldots, \norm{a_{n-1}}\}$.  Then since $C(X) \models \sigma_{n, K}$, for any $\epsilon > 0$ we can find $f \in C(X)$ with $\norm{f} \leq 2+K$ such that $\norm{a_0+a_1f+\cdots+a_{n-1}f^{n-1}+f^n} < \epsilon$, so $C(X)$ is approximately algebraically closed.
\end{proof}

\begin{fact}[{\cite[Theorem 1.1]{KawamuraMiura}}] \label{fact:FirstCountableAlgClosed}
A first-countable continuum is algebraically closed if and only if it is locally connected, has dimension at most $1$, and has $\check{H}^1(X) = 0$.
\end{fact}

\begin{thm}\label{thm:CoECAlgClosedApprox}
If $X$ is a co-existentially closed continuum then $C(X)$ is approximately algebraically closed. 
\end{thm}
\begin{proof}
The space interval $[0, 1]$ is first-countable, locally connected, dimension $1$, and has $\check{H}^1([0, 1]) = 0$.  Therefore by Fact \ref{fact:FirstCountableAlgClosed} $C([0,1])$ is algebraically closed.  It is therefore also approximately algebraically closed, so $C([0, 1]) \models \sigma_{n, K}$ for all $n$ and $K$ by Proposition \ref{prop:AlgClosedApprox}.  Each $\sigma_{n, K}$ is a $\forall\exists$-sentence, so by Lemma \ref{lem:UniversalAETheory} we have $C(X) \models \sigma_{n, K}$ for all $n$ and $K$ as well, and therefore (again by Proposition \ref{prop:AlgClosedApprox}) $C(X)$ is approximately algebraically closed.
\end{proof}

\begin{prop}\label{prop:notAlgClosed}
If $X$ is a metrizable co-existentially closed continuum then $C(X)$ is not algebraically closed.
\end{prop}
\begin{proof}
Since $X$ is metrizable it is first-countable.  By Fact \ref{fact:FirstCountableAlgClosed}, it suffices to show that $X$ is not locally connected.  By Fact \ref{fact:HIDIM}, $X$ is hereditarily indecomposable; we will show that an indecomposable continuum cannot be locally connected at any point.

Suppose to the contrary that $X$ is locally connected at a point $x$.  Let $U$ be an open set with $x \in U \subsetneq X$.  Let $C$ be a closed set with $x \in C^{\circ} \subseteq C \subseteq U$.  By local connectedness at $x$, there is a connected open set $V$ such that $x \in V \subseteq C^{\circ} \subseteq C \subseteq U$.  Then $\overline{V}$ is a proper subcontinuum of $X$ with non-empty interior; this contradicts the indecomposability of $X$, by \cite[Corollary 1.7.21]{Marcias}.
\end{proof}

We do not know if ``metrizable" can be eliminated from the hypotheses of Proposition \ref{prop:notAlgClosed}.

\begin{question}\label{q:coECAlgClosed}
Is there a co-existentially closed continuum $X$ such that $C(X)$ is algebraically closed?
\end{question}

Suppose that $X$ is a metrizable co-existentially closed continuum.  Proposition \ref{prop:AlgClosedApprox} and Theorem \ref{thm:CoECAlgClosedApprox} then imply that all models of $\op{Th}(C(X))$ are approximately algebraically closed.  In countably saturated models the sentences $\sigma_{n, K}$ express algebraic closure, not just approximate algebraic closure.  Thus if $Y$ is such that $C(Y) \equiv C(X)$ and $C(Y)$ is countably saturated then $C(Y)$ is algebraically closed, while by Proposition \ref{prop:notAlgClosed} $C(X)$ is not algebraically closed.  Therefore:

\begin{cor}\label{cor:AlgClosedNotAxiomatizable}
The property of being algebraically closed is not axiomatizable in the language of unital C*-algebras.
\end{cor}

Continuing with the setup from the paragraph above, $Y$ is also hereditarily indecomposable (because hereditary indecomposability is axiomatizable and $X$ is hereditarily indecomposable by Fact \ref{fact:HIDIM}), and hence $Y$ is not locally connected.  Thus Fact \ref{fact:FirstCountableAlgClosed} cannot be extended to non-metrizable continua in general, and the method of proof of Proposition \ref{prop:notAlgClosed} cannot be used to answer Question \ref{q:coECAlgClosed}.  There is a notion of a continuum being \emph{almost locally connected} that does follow from algebraic closure even in the non-metrizable setting (see \cite[Theorem 2.4]{Countryman}).  The same argument as above shows that this almost local connectedness is compatible with hereditary indecomposability, but we do not know if it is compatible with being co-existentially closed.

\begin{rem}
Concerning the question of whether or not $T_{\op{conn}}$ has a model companion, we observe that a negative answer to Question \ref{q:coECAlgClosed} would strongly refute the existence of such a model companion, since the existence of a model companion is equivalent to co-existential closure being preserved by co-elementary equivalence.
\end{rem}

We now return to gathering information about $K_1(C(X))$ when $X$ is co-existentially closed.  Recall that an abelian group $(G, +)$ is called \emph{$n$-divisible} (for some fixed $n \in \mathbb{N}$) if for every $g \in G$ there is $x \in G$ such that $nx=g$, where $nx$ is the sum of $n$ copies of $x$.  A group is \emph{divisible} if it is $n$-divisible for all $n \geq 1$.

\begin{fact}[{\cite[Theorem 1.3]{KawamuraMiura}}]\label{fact:Divisibility}
Suppose that $X$ is a compactum with $\dim(X) \leq 1$.  For each $n \in \mathbb{N}$, the following are equivalent:
\begin{enumerate}
    \item $\check{H}^1(X)$ is $n$-divisible,
    \item for every $f \in C(X)$ and every $\epsilon > 0$ there is $g \in C(X)$ such that $\norm{g^n-f} < \epsilon$.
\end{enumerate}
\end{fact}

It follows immediately that if $X$ is a continuum such that $C(X)$ is approximately algebraically closed and $\dim(X) \leq 1$ then $\check{H}^1(X)$ is a divisible group.  In particular, by Theorem \ref{thm:CoECAlgClosedApprox} and Fact \ref{fact:HIDIM} those hypotheses are satisfied when $X$ is co-existentially closed.  The group $\check{H}^1(X)$ is always torsion-free abelian, so we conclude:

\begin{thm}\label{thm:MainCoECCohomology}
If $X$ is a co-existentially closed continuum then $\check{H}^1(X)$ is a torsion-free divisible abelian group.
\end{thm}

\begin{rem}
In general, extracting $K$-theoretic information about a C*-algebra $A$ from $\op{Th}(A)$ is a non-trivial matter (see \cite[Sections 3.11 and 3.12]{FarahEtAl}).  By Proposition \ref{prop:K1H1} and standard translations between topological properties and C*-algebraic ones, we can view the proof of Theorem \ref{thm:MainCoECCohomology} as showing that if $T$ is the theory of commutative unital projectionless C*-algebras of real rank at most $1$, then if $A \models T$, for each $n \geq 1$ the property of having $K_1(A)$ be $n$-divisible is detected by $\op{Th}(A)$.  Although our argument depends heavily on working with continua of dimension at most $1$, we do not know if this result can be extended to more general classes of C*-algebras.
\end{rem}

For the pseudoarc $\mathbb{P}$ we have $K_1(C(\mathbb{P})) = \check{H}^1(\mathbb{P}) = 0$; this follows from the continuity of \v{C}ech cohomology, since $\mathbb{P}$ can be represented as an inverse limit of arcs, and arcs have trivial first cohomology.  In light of the fact that all co-existentially closed continua have the same $K_0$ as $\mathbb{P}$ (Corollary \ref{cor:coecK0}), and the fact that the proof of Theorem \ref{thm:MainCoECCohomology} does not appear to use the full power of having $C(X)$ be approximately algebraically closed, one might therefore conjecture that $\check{H}^1(X) = 0$ for every co-existentially closed continuum $X$.  Our next goal is to show that this is not the case even if $X$ is metrizable, and moreover (allowing $X$ to be non-metrizable) the rank of $\check{H}^1(X)$ can be arbitrarily large.  

\begin{thm}\label{thm:SameWeightK1}
Let $Y$ be a hereditarily indecomposable continuum.  There exists a co-existentially closed continuum $X$ with $w(X) = w(Y)$ and $\op{rank}(\check{H}^1(X)) \geq \op{rank}(\check{H}^1(Y))$.
\end{thm}
\begin{proof}
Using Fact \ref{fact:ContinuousImage}, let $X$ be a co-existentially closed continuum with $w(X) = w(Y)$ and with a continuous surjection $f : X \to Y$.  Every continuous surjection onto a hereditarily indecomposable continuum is confluent; this was originally proved in the metric case by Cook \cite{Cook1967}, but see \cite[Theorem 2]{BankstonSlides} for a proof in the non-metric setting.  In particular, $f$ is confluent.  It then follows by \cite[Corollary 2]{Lelek} that $f$ induces an injective homomorphism $f^* : \check{H}^1(Y) \to \check{H}^1(X)$.  As injective homomorphisms preserve independence, this gives $\op{rank}(\check{H}^1(X)) \geq \op{rank}(\check{H}^1(Y))$.
\end{proof}

\begin{cor}
There is a metrizable co-existentially closed continuum $X$ with $\check{H}^1(X) \neq 0$.
\end{cor}
\begin{proof}
Let $Y$ be a hereditarily indecomposable continuum with $\check{H}^1(Y) = \mathbb{Q}$ (one example of such is the universal pseudo-solenoid described in Section \ref{sec:PseudoSolenoid} below).  Then Theorem \ref{thm:SameWeightK1} produces a metrizable co-existentially closed continuum $X$ with $\op{rank}(\check{H}^1(X)) \geq \op{rank}(\check{H}^1(Y)) = 1$, and in particular $\check{H}^1(X) \neq 0$.
\end{proof}

By Theorem \ref{thm:SameWeightK1}, in order to show that $\check{H}^1(X)$ can have rank at least $\kappa$, for some infinite $\kappa$, it suffices to find a hereditarily indecomposable continuum $Y$ where the rank of $\check{H}^1(Y)$ is at least $\kappa$.  The $Y$ we produce will be an ultracoproduct of a metrizable continuum whose first \v{C}ech cohomology group is $\mathbb{Q}$.  We begin by proving the following, which may be of independent interest.

\begin{prop}\label{prop:Quotient}
Let $(X_i)_{i \in I}$ be a family of compacta and let $\mathcal{U}$ be an ultrafilter on an index set $I$.  Then $\prod_{\mathcal{U}}\check{H}^1(X_i)$ is a quotient of the group $\check{H}^1\left(\sum_{\mathcal{U}}X_i\right)$.
\end{prop}
\begin{proof}
It is convenient to phrase this proof categorically.  

The ultraproduct construction can be represented categorically as
\[\prod_{\mathcal{U}}\check{H}^1(X_i) = \varinjlim \left(\left(\prod_{i \in A}\check{H}^1\left(X_i\right)\right)_{A \in \mathcal{U}}, (p_{A, B})_{A \subseteq B}\right),\]
where $p_{A, B}$ is the canonical projection from $\prod_{i \in B}X_i$ to $\prod_{i \in A}X_i$ when $A \subseteq B$.

Recall that in the category of compact Hausdorff spaces the coproduct operation is taking the Stone-\v{C}ech compactification of the disjoint union of spaces, that is, for any $A \subseteq I$,
\[\coprod_{i \in A}X_i = \beta\left(\bigcup_{i\in A}(X_i \times \{i\})\right).\]
By Gel'fand duality applied to the categorical description of $\prod_{\mathcal{U}}C(X_i)$, the ultracoproduct $\sum_{\mathcal{U}}X_i$ is 
\[\sum_{\mathcal{U}}X_i = \varprojlim \left(\left(\coprod_{i \in A}X_i\right)_{A \in \mathcal{U}}, (\pi_{A, B})_{A \subseteq B}\right),\]
where for $A, B \in \mathcal{U}$ with $A \subseteq B$, the map $\pi_{A, B}$ is the natural embedding of $\coprod_{i \in A}X_i$ into $\coprod_{i \in B}X_i$.  By continuity and contravariance of \v{C}ech cohomology we therefore have that 
\[\check{H}^1\left(\sum_{\mathcal{U}}X_i\right) = \varinjlim \left(\left(\check{H}^1\left(\coprod_{i \in A}X_i\right)\right)_{A \in \mathcal{U}}, (\pi_{A, B}^*)_{A \subseteq B}\right),\]
where $\pi_{A, B}^*$ denotes the map induced on \v{C}ech cohomology by $\pi_{A, B}$.

For each $A \in \mathcal{U}$, we define a map $\eta_A : \check{H}^1(\coprod_{i \in A}X_i) \to \prod_{i \in A}\check{H}^1(X_i)$ by $\eta_A([f]) = ([f|_{X_i}])_{i \in A}$.  If $f, g \in \check{H}^1(\coprod_{i \in A}X_i)$ are such that $[f]=[g]$, i.e., $f$ and $g$ are homotopic maps, then the restriction of a homotopy from $f$ to $g$ to any $X_i$ is a homotopy from $f|_{X_i}$ to $g|_{X_i}$, and hence $([f|_{X_i}])_{i \in A} = ([g|_{X_i}])_{i \in A}$.  That is, $\eta_A$ is well-defined.  It is straightforward to verify that $\eta_A$ is a group homomorphism.  We show that $\eta_A$ is surjective.  Suppose we are given $([f_i])_{i \in A} \in \prod_{i \in A}\check{H}^1(X_i)$.  Define $f : \bigcup_{i \in A}(X_i \times \{i\}) \to \mathbb{T}$ by $f(x, i) = f_i(x)$.  Since each $f_i$ is continuous so is $f$, and therefore $f$ extends to the Stone-\v{C}ech compactification as a continuous map $f^\beta : \coprod_{i \in A}X_i \to \mathbb{T}$; we then have $\eta_A([f^\beta]) = ([f_i])_{i \in A}$.

One readily checks that the maps $(\eta_A)_{A \in \mathcal{U}}$ commute with the maps $p_{A, B}$  and $\pi^*_{A, B}$, in the sense that for all $A \subseteq B$, $p_{A, B}\eta_B = \eta_A\pi^*_{A, B}$; that is, the following diagram commutes:

\[
\begin{tikzcd}
\check{H}^1\left(\coprod_{i \in A}X_i\right) \arrow[r, "\eta_A"] 
& \prod_{i \in A}\check{H}^1(X_i) \\
\check{H}^1\left(\coprod_{i \in B}X_i\right) \arrow[r, "\eta_B"] \arrow[u, "\pi^*_{A, B}"]
& \prod_{i \in B}\check{H}^1(X_i) \arrow[u, "\rho_{A, B}"]
\end{tikzcd}
\]

Therefore the maps $(\eta_A)_{A \in \mathcal{U}}$ induce a surjective group homomorphism 
\[\varinjlim \check{H}^1\left(\coprod_{i \in A}X_i\right) \to \varinjlim \prod_{i \in A}\check{H}^1(X_i),\]
that is, we have a surjective homomorphism
\[\check{H}^1\left(\sum_{\mathcal{U}}X_i\right) \to \prod_{\mathcal{U}}\check{H}^1(X_i). \qedhere\]
\end{proof}

\begin{thm}\label{thm:LargeK1}
Given any infinite cardinal $\kappa$ there is a co-existentially closed continuum $X$ of weight $2^{\kappa}$ such that $\operatorname{rank}(\check{H}^1(X)) \geq 2^\kappa$.
\end{thm}
\begin{proof}
Let $Y$ be a metrizable hereditarily indecomposable continuum with $\check{H}^1(Y) = \mathbb{Q}$, such as the universal pseudo-solenoid described in Section \ref{sec:PseudoSolenoid} below.  Let $\mathcal{U}$ be a regular ultrafilter on $\kappa$.  Since $\mathcal{U}$ is regular and $Y$ is separable, the version of \cite[Proposition 4.3.7]{ChangKeisler} for metric structures gives that the density of $C(Y)^{\mathcal{U}}$ is $(\aleph_0)^\kappa = 2^\kappa$, so the weight of $\sum_{\mathcal{U}}Y$ is also $2^\kappa$ (this can also be proved topologically using \cite[Theorem 2.2.3]{Bankston1987}).

The hereditary indecomposability of $Y$ is an axiomatizable property of $C(Y)$, and thus is preserved by taking an ultracopower.  We may thus apply Theorem \ref{thm:SameWeightK1} to $\sum_{\mathcal{U}}Y$ to get a co-existentially closed continuum $X$ with $w(X) = w\left(\sum_{\mathcal{U}}Y\right) = 2^\kappa$ and $\op{rank}(\check{H}^1(X)) \geq \op{rank}\left(\check{H}^1\left(\sum_{\mathcal{U}}Y\right)\right)$.  In particular, $\abs{\check{H}^1\left(\sum_{\mathcal{U}}Y\right)} \leq \abs{\check{H}^1(X)}$.

Since $\prod_{\mathcal{U}}\check{H}^1(Y)$ is a quotient of $\check{H}^1(\sum_{\mathcal{U}}Y)$ (Proposition \ref{prop:Quotient}), we have $\abs{\prod_{\mathcal{U}}\check{H}^1(Y)} \leq \abs{\check{H}^1(\sum_{\mathcal{U}}Y)}$.  We chose $Y$ to have $\check{H}^1(Y) = \mathbb{Q}$, so since $\mathcal{U}$ is a regular ultrafilter on $\kappa$ we have $\abs{\prod_{\mathcal{U}}\check{H}^1(Y)} = \abs{\prod_{\mathcal{U}}\mathbb{Q}} = 2^\kappa$ (see \cite[Proposition 4.3.7]{ChangKeisler}).  Putting all this together, we have shown:
\[2^\kappa = \abs{\prod_{\mathcal{U}}\check{H}^1(Y)} \leq \abs{\check{H}^1\left(\sum_{\mathcal{U}}Y\right)} \leq \abs{\check{H}^1(X)}.\]
This completes the proof since $\check{H}^1(X)$ is a torsion-free divisible abelian group (Theorem \ref{thm:MainCoECCohomology}) and when such groups are uncountable their rank is equal to their cardinality.
\end{proof}

Having just shown that there are co-existentially closed continua with $\check{H}^1(X) \neq 0$, it is tempting to hope that models of $\op{Th}(C(\mathbb{P}))$ must have $\check{H}^1(X) = 0$, as this would give us a way to show that $\op{Th}(C(\mathbb{P}))$ is not the model companion of $T_{\op{conn}}$.  Unfortunately, this strategy does not work.

\begin{prop}\label{prop:UltrapowerK1}
If $\mathcal{U}$ is any countably incomplete ultrafilter and $X$ is any continuum with $\dim(X) = 1$ then $\check{H}^1\left(\sum_{\mathcal{U}}X\right) \neq 0$.
\end{prop}
\begin{proof}
By \cite[Theorem 2.1]{BankstonEtAl} a compactum is $3$-chainable if and only if it is a one-dimensional continuum with trivial first \v{C}ech cohomology group.  Every ultracoproduct of compacta by a countably incomplete ultrafilter fails to be $3$-chainable \cite[Lemma 5.3]{BankstonEtAl}.  Since being a continuum and being one-dimensional are elementary properties they are preserved by ultracoproducts.  Thus $\sum_{\mathcal{U}}X$ is a one-dimensional continuum that is not $3$-chainable, and hence $\check{H}^1\left(\sum_{\mathcal{U}}X\right) \neq 0$.
\end{proof}

\begin{cor}
There are models of $\op{Th}(C(\mathbb{P}))$ with non-zero $K_1$.
\end{cor}

\begin{cor}
The connected component of the unitary group is not definable in $\op{Th}(C(\mathbb{P}))$.
\end{cor}
\begin{proof}
Since $\dim(\mathbb{P}) = 1$, the C*-algebra $C(\mathbb{P})$ has stable rank $1$ (see the proof of Proposition \ref{prop:K1H1} above).  There is a natural group homomorphism $\phi : \mathcal{U}(C(\mathbb{P}))/U_0(C(\mathbb{P})) \to K_1(C(\mathbb{P}))$; since we are working with a commutative C*-algebra this map is injective \cite[Proposition 8.3.1]{KTheoryBook}, and since $C(\mathbb{P})$ has stable rank $1$ the map is surjective \cite[Theorem 10.10]{Rieffel}.  Being abelian and having stable rank $1$ are elementary properties, so if $M \models \op{Th}(C(\mathbb{P}))$ then $K_1(M) \cong \mathcal{U}(M)/\mathcal{U}_0(M)$.  

As a consequence of the above, if the connected component of the unitary group is definable then $K_1(M)$ is in $M^{\operatorname{eq}}$. It then follows from \cite[Proposition 3.12.1(iii)]{FarahEtAl} that if $B$ is any C*-algebra with $C(\mathbb{P}) \preceq B$, then $K_1(C(\mathbb{P})) \preceq K_1(B)$.  However, $K_1(C(\mathbb{P})) = 0$, while Proposition \ref{prop:UltrapowerK1} (along with Proposition \ref{prop:K1H1}) shows that ultrapowers of $C(\mathbb{P})$ often have non-zero $K_1$, giving the desired contradiction.
\end{proof}

For planar continua we can improve Theorem \ref{thm:MainCoECCohomology}.

\begin{thm}
If $X$ is a planar co-existentially closed continuum then $\check{H}^1(X) = 0$, and hence $X$ does not separate the plane.    
\end{thm}
\begin{proof}
By \cite[Theorem 3.3]{Krasinkiewicz}, $\check{H}^1(X)$ is a finitely divisible group, meaning that if $g \in \check{H}^1(X)$ and $g \neq 0$ then $\{n \in \mathbb{N} : g \text{ is divisible by $n$}\}$ is finite (note, in particular, that on the definition given in \cite{Krasinkiewicz} the trivial group is finitely divisble).  On the other hand, Theorem \ref{thm:MainCoECCohomology} shows that $\check{H}^1(X)$ is a divisible group, meaning every element is divisible by every $n \in \mathbb{N}$.  This combination of properties is possible only when $\check{H}^1(X) = 0$.  The claim about not separating the plane then follows directly from \cite[Theorem 1]{Lau}.
\end{proof}

There are $2^{2^{\aleph_0}}$ distinct hereditarily indecomposable one-dimensional plane continua that do not separate the plane (see \cite[Theorem 2]{Ingram}), and moreover many of these are not continuous images of the pseudoarc.  We do not know if any of these are co-existentially closed.

\begin{prob}\label{prob:Planar}
Is there a planar co-existentially closed continuum other than the pseudoarc?
\end{prob}

\section{Continua that are not co-existentially closed}
As an application of our results on $K_1(C(X))$ for co-existentially closed continua $X$, we show that various continua cannot be co-existentially closed.  The main result of this section is that the only pseudo-solenoid that could be co-existentially closed is the universal one.  The continua considered in this section will all be metrizable, but since that was not the case in earlier sections we will include the metrizability hypothesis in our theorem statements.  

Recall that if $C$ is a continuum then a continuum $X$ is \emph{$C$-like} if it can be written as an inverse limit of copies of $C$.  In particular, \emph{arc-like} continua (also known as the metrizable \emph{chainable} continua) are inverse limits of copies of $[0, 1]$, while \emph{circle-like continua} (also known as the metrizable \emph{circularly chainable} continua) are inverse limits of copies of the circle $\mathbb{T}$.

\subsection*{Solenoids}
Consider the circle as $\mathbb{T} = \{z \in \mathbb{C} : \abs{z} = 1\}$.  For each $n \in \mathbb{N}$, let $\mu_n : \mathbb{T} \to \mathbb{T}$ be the map $\mu_n(z) = z^n$.  

\begin{defn}
A \emph{solenoid} is a metrizable continuum that is not arc-like and that is circle-like where each map in the inverse system is a map of the form $\mu_n$ (with possibly different values of $n$ for different maps).
\end{defn}

\begin{defn}
Let $N = (n_1, n_2, \ldots)$ be a sequence of positive natural numbers.  The \emph{$N$-adic solenoid} $\mathbb{S}_N$ is the inverse limit of the system
\[\mathbb{T} \xleftarrow{\mu_{n_1}} \mathbb{T} \xleftarrow{\mu_{n_2}} \cdots.\]
\end{defn}

It is immediate from the definition that every solenoid is the $N$-adic solenoid for some $N$.   
\begin{defn}
    Let $(n_1, n_2, \ldots)$ be a sequence of positive integers.  The corresponding \emph{supernatural number} is the formal product $\prod_{i=1}^{\infty}n_i$.  We say that two supernatural numbers $N$ and $M$ are \emph{equivalent}, and write $N \sim M$, if there are finite values $1 \leq n_0, m_0 < \infty$ such that $m_0\cdot N = n_0 \cdot M$.
\end{defn}

Note that any supernatural number can be expressed as a product of the form $\prod_{i \in \mathbb{N}}p_i^{k_i}$, where $p_i$ is the $i$th prime and $k_i \in \{0, 1, 2, \ldots, \infty\}$.  The equivalence of two supernatural numbers means that these formal products differ on at most finitely many primes, each of which appears to a finite power (see the discussion before and after \cite[Definition 1.1]{HurderLukina} for more details).  We denote by $[\infty]$ the supernatural number $[\infty] = \prod_{i \in \mathbb{N}}p_i^\infty$.

Bing \cite{Bing} showed that if $N = (n_1, n_2, \ldots)$ and $M = (m_1, m_2, \ldots)$ are two sequences that produce equivalent supernatural numbers, then $\mathbb{S}_N \cong \mathbb{S}_M$.  McCord \cite{McCord} showed the converse.  Thus solenoids are classified up to homeomorphism by equivalence of supernatural numbers:

\begin{fact}[{\cite{Bing}, \cite{McCord}}]\label{fact:SameSupernatural}
Let $N$ and $M$ be supernatural numbers.  Then $\mathbb{S}_N \cong \mathbb{S}_M$ if and only if $N \sim M$.
\end{fact}

Consequently, we often index solenoids by supernatural numbers instead of by sequences of natural numbers.  

Every solenoid contains an arc (in fact, every proper subcontinuum of a solenoid is an arc) and therefore solenoids are not hereditarily indecomposable.  Since every co-existentially closed continuum is hereditarily indecomposable (Fact \ref{fact:HIDIM}) it follows that no solenoid can be co-existentially closed.  Nevertheless, solenoids provide useful information about the class of pseudo-solenoids (described below), which are hereditarily indecomposable (but, as we will show, still not co-existentially closed).  The fact we will need is the following, which follows easily from the continuity of \v{C}ech cohomology, or can be extracted from the proof of \cite[Theorem 5.2]{Fearnley1969}.
\begin{fact}\label{fact:SolenoidDivisible}
Let $N$ be a supernatural number and $X = \mathbb{S}_N$.  For each $k \in \mathbb{N}$, the cohomology group $\check{H}^1(X)$ is $k$-divisible if and only if $k^\infty$ divides $N$.
\end{fact}

\begin{rem}\label{rem:SolenoidTheory}
Variations of the sentences $\sigma_{n,K}$ (used in Proposition \ref{prop:AlgClosedApprox}) can be used to express that $C(X)$ is approximately closed under $n$th roots.  Thus, by Fact \ref{fact:Divisibility}, when $\dim(X) \leq 1$ the property of having $\check{H}^1(X)$ be $n$-divisible is $\forall\exists$-axiomatizable.  Moreover, by Fact \ref{fact:SolenoidDivisible} this means that if $X = \mathbb{S}_N$ then $\op{Th}(C(X))$ can detect for which $n$ we have $n^{\infty} \mid N$. Thus if $C(\mathbb{S}_N) \equiv C(\mathbb{S}_{M})$ then $N$ and $M$ are infinitely divisible by the same natural numbers.  In particular, this implies that $C(\mathbb{S}_{[\infty]})$ is the only model of its theory whose spectrum is a solenoid.
\end{rem}

\subsection*{Pseudo-solenoids}\label{sec:PseudoSolenoid}
\begin{defn}
A \emph{pseudo-solenoid} is a circle-like, non-arc-like, hereditarily indecomposable metrizable continuum.
\end{defn}

The construction of pseudo-solenoids is similar to both the constructions of solenoids (in that maps winding the circle around itself are used) and the construction of the pseudoarc (in that the maps become increasingly ``crooked").  In fact, the precise quantitative details of how the maps are made to be crooked does not affect the resulting homeomorphism type of the pseudo-solenoid, and pseudo-solenoids are classified in the same manner as solenoids.  

\begin{fact}[{\cite{Fearnley}; see also \cite[Section 5.1]{Kubis}}]
To each pseudo-solenoid there is an associated supernatural number, and two pseudo-solenoids are homeomorphic if and only if they have equivalent supernatural numbers.
\end{fact}

\begin{thm}\label{thm:PseudoSolenoidCoEC}
Let $X$ be a pseudo-solenoid, and suppose that there is a prime $p$ such that $p$ appears with a finite power in the supernatural number associated to $X$.  Then $X$ is not co-existentially closed.
\end{thm}
\begin{proof}
Let $N$ be the supernatural number associated to $X$.  As shown in the proof of \cite[Theorem 3.3]{Fearnley}, $X$ has the same first \v{C}ech cohomology as the solenoid $\mathbb{S}_N$ associated to $N$.  That is, $\check{H}^1(X) \cong \check{H}^1(\mathbb{S}_{N})$.  Let $p$ be a prime that appears only finitely often in $N$.  Then Fact \ref{fact:SolenoidDivisible} tells us that $\check{H}^1(X)$ is not $p$-divisible, and hence $X$ is not co-existentially closed by Theorem \ref{thm:MainCoECCohomology}.
\end{proof}

Theorem \ref{thm:PseudoSolenoidCoEC} implies that the only pseudo-solenoid that could be co-existentially closed is the so-called \emph{universal} pseudo-solenoid, which is the pseudo-solenoid $X$ associated to the supernatural number $[\infty]$; the universal pseudo-solenoid has $\check{H}^1(X) \cong \mathbb{Q}$.  We ask in Problem \ref{prob:universal} below whether the universal pseudo-solenoid is co-existentially closed.

\begin{rem}
    In the proof of Theorem \ref{thm:PseudoSolenoidCoEC} we noted that divisibility of $\check{H}^1(X)$, for $X$ a pseudo-solenoid, depends on the supernatural number of $X$ in the same way as is the case for solenoids.  As a consequence, the observations in Remark \ref{rem:SolenoidTheory} apply to pseudo-solenoids as well.  We also note that if $X$ is a solenoid and $Y$ is a pseudo-solenoid with the same supernatural number as $X$ then we still have $C(X) \not\equiv C(Y)$, since hereditary indecomposability is axiomatizable.
\end{rem}

\subsection*{Homogeneous continua}
In this section we apply classification results about homogeneous continua to show that several types of homogeneous metrizable continua cannot be co-existentially closed.  We need to recall one more class of continua.

\begin{defn}
    A \emph{solenoid of pseudoarcs} is a circle-like continuum that admits a continuous decomposition into pseudoarcs with decomposition space homeomorphic to $\mathbb{T}$.
\end{defn}

The original example of a solenoid of pseudoarcs is the \emph{circle of pseudoarcs} introduced by Bing and Jones \cite{BingJones}, while the general class of solenoids of pseudoarcs was introduced by Rogers \cite{Rogers}.  Any solenoid of pseudoarcs is decomposable, and hence not co-existentially closed.  This class is of interest to us because it arises in several classification theorems regarding homogeneous continua, which we now apply to show that many homogeneous continua are not co-existentially closed.

\begin{thm}
    The only co-existentially closed, circle-like, homogeneous, metrizable continuum is the pseudoarc.
\end{thm}
\begin{proof}
Hagopian and Rogers \cite{HagopianRogers} have shown that the only non-degenerate, circle-like, homogeneous, metrizable continua are the pseudoarc, the solenoids, and the solenoids of pseudoarcs.  We have already observed that the latter two types of continua are not co-existentially closed.
\end{proof}

Using a classification result of Hoehn and Oversteegen we can give the following partial result towards Problem \ref{prob:Planar}.

\begin{thm}\label{homogeneousPlanar}
    The only homogeneous planar continuum that is co-existentially closed is the pseudoarc.
\end{thm}

\begin{proof}
The main result of \cite{HoehnOversteegen} is that the only non-degenerate homogeneous subcontinua of the plane are the circle, the circle of pseudoarcs, and the pseudoarc.  The circle and the circle of pseudoarcs are decomposable, hence not co-existentially closed.
\end{proof}

Despite the variety of known constructions of continua, and the fact that many metrizable co-existentially closed continua exist, we nevertheless still have the following question:

\begin{prob}\label{prob:Concrete}
Give a concrete example of a metrizable co-existentially closed continuum other than the pseudoarc.
\end{prob}

Aside from the properties of co-existentially closed continua already mentioned in this paper, we also observe that any co-existentially closed continuum must be unicoherent (i.e., have the property that if $A$ and $B$ are subcontinua of $X$ and $A \cup B = X$ then $A \cap B$ is connected).  This follows from Lemma \ref{lem:UniversalAETheory} because unicoherence is a $\forall\exists$-axiomatizable property (see \cite[Remark 5.7(iii)]{Bankston2006}).

As a specific example of Problem \ref{prob:Concrete}, we believe that the answer to the following question is likely negative:

\begin{prob}\label{prob:universal}
    Is the universal pseudo-solenoid co-existentially closed?
\end{prob}

We note that the results concerning homogeneous continua do not address this question because pseudo-solenoids are not homogeneous \cite{Sturm}, and the observation about unicoherence also does not apply because indecomposable circle-like continua are unicoherent (see, e.g., \cite{Hagopian}).

\bibliographystyle{amsalpha}
\bibliography{main}

\newcommand{\etalchar}[1]{$^{#1}$}
\providecommand{\bysame}{\leavevmode\hbox to3em{\hrulefill}\thinspace}
\providecommand{\MR}{\relax\ifhmode\unskip\space\fi MR }
\providecommand{\MRhref}[2]{%
  \href{http://www.ams.org/mathscinet-getitem?mr=#1}{#2}
}
\providecommand{\href}[2]{#2}
\begin{thebibliography}{BYBHU08}

\bibitem[Ban87]{Bankston1987}
Paul Bankston, \emph{Reduced coproducts of compact {H}ausdorff spaces}, J.
  Symbolic Logic \textbf{52} (1987), no.~2, 404--424. \MR{890449}

\bibitem[Ban99]{Bankston1999}
\bysame, \emph{A hierarchy of maps between compacta}, J. Symbolic Logic
  \textbf{64} (1999), no.~4, 1628--1644. \MR{1780075}

\bibitem[Ban00]{Bankston2000}
\bysame, \emph{Some applications of the ultrapower theorem to the theory of
  compacta}, Appl. Categ. Structures \textbf{8} (2000), no.~1-2, 45--66, Papers
  in honour of Bernhard Banaschewski (Cape Town, 1996). \MR{1785837}

\bibitem[Ban06]{Bankston2006}
\bysame, \emph{The {C}hang-\l o\'{s}-{S}uszko theorem in a topological
  setting}, Arch. Math. Logic \textbf{45} (2006), no.~1, 97--112. \MR{2209739}

\bibitem[Ban14]{BankstonSlides}
P.~Bankston, \emph{When semi-monotone implies monotone}, Slides from {O}xford
  {T}opology {S}eminar,
  \url{https://epublications.marquette.edu/cgi/viewcontent.cgi?article=1219&context=mscs_fac},
  2014.

\bibitem[BBRR06]{BankstonEtAl}
Taras Banakh, Paul Bankston, Brian Raines, and Wim Ruitenburg,
  \emph{Chainability and {H}emmingsen's theorem}, Topology Appl. \textbf{153}
  (2006), no.~14, 2462--2468. \MR{2243725}

\bibitem[Bin51]{Bing1951}
R.~H. Bing, \emph{Concerning hereditarily indecomposable continua}, Pacific J.
  Math. \textbf{1} (1951), 43--51. \MR{43451}

\bibitem[Bin60]{Bing}
\bysame, \emph{A simple closed curve is the only homogeneous bounded plane
  continuum that contains an arc}, Canadian J. Math. \textbf{12} (1960),
  209--230. \MR{111001}

\bibitem[BJ59]{BingJones}
R.~H. Bing and F.~B. Jones, \emph{Another homogeneous plane continuum}, Trans.
  Amer. Math. Soc. \textbf{90} (1959), 171--192. \MR{100823}

\bibitem[BK22]{Kubis}
A.~Barto\v{s} and W.~Kubi\'s, \emph{Hereditarily indecomposable continua as
  generic mathematical structures}, arXiv:2208.06886 (2022).

\bibitem[Bla06]{Blackadar}
B.~Blackadar, \emph{Operator algebras}, Encyclopaedia of Mathematical Sciences,
  vol. 122, Springer-Verlag, Berlin, 2006, Theory of $C^*$-algebras and von
  Neumann algebras, Operator Algebras and Non-commutative Geometry, III.
  \MR{2188261}

\bibitem[BYBHU08]{BenYaacovEtAl}
Ita\"{\i} Ben~Yaacov, Alexander Berenstein, C.~Ward Henson, and Alexander
  Usvyatsov, \emph{Model theory for metric structures}, Model theory with
  applications to algebra and analysis. {V}ol. 2, London Math. Soc. Lecture
  Note Ser., vol. 350, Cambridge Univ. Press, Cambridge, 2008, pp.~315--427.
  \MR{2436146}

\bibitem[CK90]{ChangKeisler}
C.~C. Chang and H.~J. Keisler, \emph{Model theory}, third ed., Studies in Logic
  and the Foundations of Mathematics, vol.~73, North-Holland Publishing Co.,
  Amsterdam, 1990. \MR{1059055}

\bibitem[Coo67]{Cook1967}
H.~Cook, \emph{Continua which admit only the identity mapping onto
  non-degenerate subcontinua}, Fund. Math. \textbf{60} (1967), 241--249.
  \MR{220249}

\bibitem[Cou67]{Countryman}
R.~S. Countryman, Jr., \emph{On the characterization of compact {H}ausdorff
  {$X$} for which {$C(X)$} is algebraically closed}, Pacific J. Math.
  \textbf{20} (1967), 433--448. \MR{208410}

\bibitem[EGV16]{EagleGoldbringVignati}
Christopher~J. Eagle, Isaac Goldbring, and Alessandro Vignati, \emph{The
  pseudoarc is a co-existentially closed continuum}, Topology Appl.
  \textbf{207} (2016), 1--9. \MR{3501261}

\bibitem[Fea69]{Fearnley1969}
L.~Fearnley, \emph{Embeddings of topological products of sphere-like continua},
  Proc. London Math. Soc. (3) \textbf{19} (1969), 586--600. \MR{259865}

\bibitem[Fea72]{Fearnley}
Lawrence Fearnley, \emph{Classification of all hereditarily indecomposable
  circularly chainable continua}, Trans. Amer. Math. Soc. \textbf{168} (1972),
  387--401. \MR{296903}

\bibitem[FHL{\etalchar{+}}21]{FarahEtAl}
Ilijas Farah, Bradd Hart, Martino Lupini, Leonel Robert, Aaron Tikuisis,
  Alessandro Vignati, and Wilhelm Winter, \emph{Model theory of {$\rm
  C^*$}-algebras}, Mem. Amer. Math. Soc. \textbf{271} (2021), no.~1324,
  viii+127. \MR{4279915}

\bibitem[Hag77]{Hagopian}
Charles~L. Hagopian, \emph{Homogeneous circle-like continua that contain
  pseudo-arcs}, Topology {P}roceedings, {V}ol. {I} ({C}onf., {A}uburn {U}niv.,
  {A}uburn, {A}la., 1976), Ohio University, Institute for Medicine and
  Mathematics, Athens, OH, 1977, pp.~29--32. \MR{458380}

\bibitem[Har07]{Hart}
Klaas~Pieter Hart, \emph{There is no categorical metric continuum}, Continuum
  theory: in honor of {P}rofessor {D}avid {P}. {B}ellamy on the ocassion of his
  60th birthday, Aportaciones Mat. Investig., vol.~19, Soc. Mat. Mexicana,
  M\'{e}xico, 2007, pp.~39--43. \MR{2743834}

\bibitem[Hat02]{Hatcher}
Allen Hatcher, \emph{Algebraic topology}, Cambridge University Press,
  Cambridge, 2002. \MR{1867354}

\bibitem[HL21]{HurderLukina}
S.~Hurder and O.~Lukina, \emph{The prime spectrum of solenoidal manifolds},
  arXiv:2103.06825 (2021).

\bibitem[HO16]{HoehnOversteegen}
Logan~C. Hoehn and Lex~G. Oversteegen, \emph{A complete classification of
  homogeneous plane continua}, Acta Math. \textbf{216} (2016), no.~2, 177--216.
  \MR{3573330}

\bibitem[HR77]{HagopianRogers}
Charles~L. Hagopian and James~T. Rogers, Jr., \emph{A classification of
  homogeneous, circle-like continua}, Houston J. Math. \textbf{3} (1977),
  no.~4, 471--474. \MR{464194}

\bibitem[Ing79]{Ingram}
W.~T. Ingram, \emph{Hereditarily indecomposable tree-like continua}, Fund.
  Math. \textbf{103} (1979), no.~1, 61--64. \MR{535836}

\bibitem[KM07]{KawamuraMiura}
Kazuhiro Kawamura and Takeshi Miura, \emph{On the existence of continuous
  (approximate) roots of algebraic equations}, Topology Appl. \textbf{154}
  (2007), no.~2, 434--442. \MR{2278692}

\bibitem[Kna22]{Knaster}
B.~Knaster, \emph{Un continu dont tout sous-continu est indecomposable},
  Fundamenta Mathematicae \textbf{3} (1922), 247–286.

\bibitem[Kra76]{Krasinkiewicz}
J.~Krasinkiewicz, \emph{Mappings onto circle-like continua}, Fund. Math.
  \textbf{91} (1976), no.~1, 39--49. \MR{407816}

\bibitem[Lau80]{Lau}
A.~Lau, \emph{Plane continua and transformation groups}, Proc. Amer. Math. Soc.
  \textbf{78} (1980), no.~4, 608--610. \MR{556642}

\bibitem[Lel66]{Lelek}
A.~Lelek, \emph{On confluent mappings}, Colloq. Math. \textbf{15} (1966),
  223--233. \MR{208574}

\bibitem[Mac18]{Marcias}
Sergio Mac\'{\i}as, \emph{Topics on continua}, second ed., Springer, Cham,
  2018. \MR{3823258}

\bibitem[Mar60]{mardesic}
Sibe Marde\v{s}i\'{c}, \emph{On covering dimension and inverse limits of
  compact spaces}, Illinois J. Math. \textbf{4} (1960), 278--291. \MR{116306}

\bibitem[McC65]{McCord}
M.~C. McCord, \emph{Inverse limit sequences with covering maps}, Trans. Amer.
  Math. Soc. \textbf{114} (1965), 197--209. \MR{173237}

\bibitem[Nad92]{Nadler}
Sam~B. Nadler, Jr., \emph{Continuum theory}, Monographs and Textbooks in Pure
  and Applied Mathematics, vol. 158, Marcel Dekker, Inc., New York, 1992, An
  introduction. \MR{1192552}

\bibitem[Rie83]{Rieffel}
Marc~A. Rieffel, \emph{Dimension and stable rank in the {$K$}-theory of
  {$C\sp{\ast}$}-algebras}, Proc. London Math. Soc. (3) \textbf{46} (1983),
  no.~2, 301--333. \MR{693043}

\bibitem[Rog77]{Rogers}
James~T. Rogers, Jr., \emph{Solenoids of pseudo-arcs}, Houston J. Math.
  \textbf{3} (1977), no.~4, 531--537. \MR{464193}

\bibitem[RrLL00]{KTheoryBook}
M.~R\o~rdam, F.~Larsen, and N.~Laustsen, \emph{An introduction to {$K$}-theory
  for {$C^*$}-algebras}, London Mathematical Society Student Texts, vol.~49,
  Cambridge University Press, Cambridge, 2000. \MR{1783408}

\bibitem[Stu14]{Sturm}
Frank Sturm, \emph{Pseudo-solenoids are not continuously homogeneous}, Topology
  Appl. \textbf{171} (2014), 71--86. \MR{3207489}

\end{thebibliography}
\end{document}